\documentclass[11pt]{amsart}

\usepackage{amssymb,amsthm,amsmath}
\usepackage{hyperref}
\usepackage[alphabetic]{amsrefs}
\usepackage[all]{xy}
\usepackage{enumerate}

\usepackage[margin=1.25in]{geometry}

\newtheorem{theorem}{Theorem}[section]
\newtheorem{corollary}[theorem]{Corollary}
\newtheorem{lemma}[theorem]{Lemma}
\newtheorem{proposition}[theorem]{Proposition}

\theoremstyle{remark}

\numberwithin{equation}{section}

\newcommand{\N}{\mathbb{N}}
\newcommand{\T}{\mathbb{T}}

\newcommand{\C}{\mathbb{C}}

\newcommand{\Rat}{\mathcal{R}}
\newcommand{\ord}{\mathrm{ord}}

\renewcommand{\epsilon}{\varepsilon}
\renewcommand{\leq}{\leqslant}
\renewcommand{\geq}{\geqslant}

\title[Similarity of  perturbations of the shift...]{
Similarity of  perturbations of the shift and a different product of rational functions}
\author{Leonel Robert}
\begin{document}
\maketitle

\begin{abstract} 
Necessary and sufficient conditions are given for the similarity between two perturbations of the (backward) shift by rank one operators, under certain assumptions on the perturbations. The proof of similarity is based on an explicit  construction of intertwiners between the  perturbations. These intertwiners, in turn, are  parametrized by the elements of a certain algebra, with the group of ``circle invertible" elements of this algebra giving rise to invertible intertwiners.
\end{abstract}

\section{Introduction}
Let $H^2$ denote the Hardy space of the circle. Let $U\colon H^2\to H^2$ denote the backward shift operator.
The perturbations of $U$ (or $U^*$) by a rank one operator  have been occasionally studied  and shown to have a rich theory  (see \cite{clark}, \cite{nakamura}, \cite{cassier-timotin}). It is shown in \cite{robertsim} that a large class of perturbations of $U$ by small and/or compact operators  are in fact similar to  perturbations of $U$ by rank one operators. This  motivates the question addressed in this paper:  when are two perturbations of $U$ by rank one operators similar? Theorem \ref{main} below gives necessary and sufficient conditions for the operators 
$U+r\otimes \phi$ and $U+s\otimes\phi$ to be similar, under the assumption that $r$ and $s$ are rational functions
in $H^2$ (i.e., with poles outside of the closed  unit disc). Although assuming that  $r$ and $s$
are rational certainly simplifies the analysis, we will see how even in this case an interesting algebraic structure remains. The unraveling of this structure leads to the solution of  the similarity problem.

Let us
introduce some notation. Let $D$ denote the closed unit disc. Let $\Rat(D)$ denote the rational functions with poles
outside $D$. Let $\phi\in H^2$ and $r\in \Rat(D)$. For each $|w|<1$, define 
\begin{align*}
\Gamma_+(w;r) =w\langle r,\frac{\phi}{1-\overline{w}z}\rangle,\mbox{ \quad }
\Gamma_-(w;r) =\langle \frac{\phi}{w-z},r\rangle.
\end{align*}
These functions are analytic in $w$. It will be shown below that $\Gamma_+(\cdot,r)$ is a rational function with poles outside $D$; in particular, it extends analytically to a neighborhood of $D$. 
Let $\ord_w(f)$ denote the order of the zero of $f$ on $w$, where $f$ is analytic in a neighborhood
of $w$.

\begin{theorem}\label{main}
Let $r,s\in \Rat(D)$.
The following propositions are equivalent:
\begin{enumerate}[(i)]
\item
The operators $U+r\otimes \phi$ and $U+s\otimes \phi$ are similar.
\item
\begin{enumerate}
\item
For each $|w|\leq 1$, $\ord_w(1-\Gamma_+ (w;r))=\ord_w(1-\Gamma_+(w;s))$, and 
\item
for each $|w|<1$,
\[\min(\ord_w(\phi),\ord_w(1-\Gamma_- (w;r)))=\min(\ord_w(\phi),\ord_w(1-\Gamma_-(w;s))).\]
\end{enumerate}
\end{enumerate}
\end{theorem} 
The proof of (ii) $\Rightarrow$ (i) relies on the construction of certain  intertwiners between the operators $U+r\otimes \phi$, with $r$ varying over $\Rat(D)$ and $\phi\in H^2$ fixed. In turn, these intertwiners 
are described in terms of a ``twisted" multiplication on the rational functions. More specifically, 
define on $\Rat(D)$ 
the binary operation
\[
r\times s=zrT_{\overline{\phi}}(s)+zsT_{\overline{\phi}}r-T_{\overline{\phi}}(zrs).
\]
Here $z\colon \T\to\T$ denotes the identity function and $T_{\overline{\phi}}$ is the co-analytic Toeplitz operator with symbol $\overline{\phi}$. It is easy to check that $r\times s$ is again an element of $\Rat(D)$. It is not at all clear that the operation
$\times$ is associative, but it will be shown below that this is  the case (Section \ref{intertwiners}). Thus, $\Rat(D)$ becomes an algebra under the multiplication $\times$ (and standard addition and scalar multiplication).  For each $r\in \Rat(D)$, define $K_r\colon H^2\to H^2$ by $K_r f=r\times f$. The operators $I-K_r$
are intertwiners between perturbations of $U$: 
\[
(I-K_r)(U+s\otimes \phi)=(U+(r\circ s)\otimes \phi)(I-K_r).
\]
Here $r\circ s:=r+s-r\times s$ is the operation of circle composition in the algebra $(\Rat(D),\times)$. The proof of Theorem \ref{main} (ii) $\Rightarrow$ (i)  passes through an analysis of the algebra $(\Rat(D),\times)$
and in particular of its circle invertible elements (Theorem \ref{th:timesalgebra} and Corollary \ref{cor:circlegroup}).
On the other hand, the implication (i) $\Rightarrow$ (ii)  follows from a rather straightforward spectral analysis (Section \ref{pointspectrum}).

\section{Intertwiners}\label{intertwiners}
Let us start by fixing some notation. For each $|w|<1$, let $k_w=1/(1-\overline{w}z)$. 
If $f\in L_2(\T)$, we shall always understand by $f(w)$ the evaluation on $w$ of the harmonic extension of $f$ to $D$, i.e., $f(w)=\langle fk_w,k_w\rangle $.

Let $P_+\colon L_2(\T)\to H^2$ denote the orthogonal projection.
We denote by $T_f$ the Toeplitz operator on $H^2$ with symbol $f\in L_2(\T)$, i.e.,
$T_fg=P_+(fg)$. If $f$ is unbounded then   $T_f$ is only densely defined (say, on $\Rat(D)$). However,  in this case we will  find it useful  to regard $T_f$ as a continuous operator from $H^2$ to the space $\mathcal H(D)$ of analytic functions in the interior of $D$, endowed with the topology of uniform convergence on compact sets (see \cite[(IV-12)]{sarason}).  The operator $P_+\colon L_2(\T)\to \mathcal H(D)$ is then taken to mean 
$P_+(f)(w):=\langle f,k_w\rangle$, for $|w|<1$.

For $|w|<1$ and $n=0,1,\dots$, let
\[
k_w^{(n)}=\frac{n!z^n}{(1-\overline{w}z)^{n+1}}.
\]
Observe that $k_w^{(n)}=\frac{d^n}{d\overline w^n} k_w$, i.e., $k_w^{(n)}$ is the $n$-th derivative of $k_w$ with respect to the parameter $\overline{w}$. 
Let $\mathcal S_w$ denote the linear span of  $\{k_w^{(0)},k_w^{(1)},\dots\}$. The decomposition of a rational function into simple fractions implies that
\[
\Rat(D)=\bigoplus_{|w|<1}\mathcal S_w.
\]

Let $\phi\in H^2$ and $|w|<1$. We have the following
formula for evaluating the Toeplitz operator $T_{\overline\phi}$ on $k_w^{(n)}$:
\begin{align}\label{evalTphi}
T_{\overline\phi}(k_w^{(n)})=\frac{d^n}{d\overline w^n}(\overline{\phi(w)}k_w).
\end{align}
This formula is deduced from the case $n=0$--which is well known--by repeatedly differentiating
with respect to $\overline{w}$. From this formula
we see that $\mathcal S_w$ and $\Rat(D)$ are both invariant by $T_{\overline\phi}$. It follows that
if $r,s\in \Rat(D)$ then 
\[
r\times s:=zrT_{\overline\phi}(s)+zs T_{\overline\phi}(r)-T_{\overline\phi}(zrs)
\]
is also in $\Rat(D)$.

Let $r\in \Rat(D)$. Define $K_r f=r \times f$, with $f\in H^2$. Observe that we can make sense of $r\times f$
as a function in $\mathcal H(D)$, bearing in mind the convention stated above for the evaluation of Toeplitz operators with unbounded symbol. Nevertheless, we will show shortly that $K_r$ is in fact a bounded operator on $H^2$.

For each $|w|<1$ and $f\in H^2$, let $\Gamma_-(w;f):=\langle  \frac{\phi}{ w-z},f\rangle$.
\begin{lemma}\label{lemma:deltaformula}
Let $|w|<1$, $n\in \N$, and $f\in H^2$. Then
\[
k_w^{(n)}\times f=zT_{\overline{\phi}}(k_w^{(n)})f+
\frac{d^m}{d\overline w^m}(\overline{\Gamma_-( w;f)}\cdot k_w).
\] 
\end{lemma}

\begin{proof}
We have 
\begin{align*}
k_w^{(n)} \times f &=zT_{\overline{\phi}}(k_w^{(n)})f+zk_w^{(n)}T_{\overline{\phi}}(f)-T_{\overline\phi}(zfk_w^{(n)}).
\end{align*}
The first term on the right hand side is already present in the desired formula. Thus, we must deal with the other two terms. 
We have  
\[
zk_w^{(n)}T_{\overline{\phi}}(f)-T_{\overline\phi}(zfk_w^{(n)})=zk_w^{(n)}P_+(\overline{\phi}f)  - P_+(\overline{\phi}zk_w^{(n)}f)=P_+(k_w^{(n)}z(P_+ - I)(\overline{\phi}f)).
\] 
Set $z(P_+ - I)(\overline{\phi}f)=\tilde f$, so that $zk_w^{(n)}T_{\overline{\phi}}(f)-T_{\overline\phi}(zfk_w^{(n)})=P_+(k_w^{(n)}\tilde f)$. Observe that $\tilde f \perp zH^2$. So,
the harmonic extension of $\tilde f$ to the unit disc is  conjugate analytic, i.e., analytic in $\overline w$. By the same argument used in the derivation of \eqref{evalTphi}, we have $P_+(k_w^{(n)}\tilde f)=\frac{d^n}{d\overline w^n}(\tilde f(w)k_w)$.
On the other hand,
\[
\tilde f(w)=\langle \tilde f,\frac{1}{1-\overline z w}\rangle
=\langle z(P_+ - I)(\overline{\phi}f),\frac{1}{1-\overline z w}\rangle=
\langle f,\frac{\phi}{w-z}\rangle=\overline {\Gamma_-(w;f)}.
\] 
This proves the lemma.
\end{proof}

\begin{proposition}\label{prop:Kr}
$K_r\colon H^2\to H^2$ is a bounded operator  which is a perturbation of the analytic Toeplitz operator with symbol 
$zT_{\overline{\phi}}r$ by a finite rank operator. 
\end{proposition}

\begin{proof}
We may reduce ourselves to the case that $r=k_w^{(n)}$ for some $|w|<1$ and $n\in \N$ (since these functions span $\Rat(D)$).
In this case, the proposition follows from the previous lemma. Indeed, observe that $f\mapsto zT_{\overline{\phi}}(k_w^{(n)})f$
is a Toeplitz operator with symbol $zT_{\overline{\phi}}(k_w^{(n)})$ and that $f\mapsto \frac{d^i}{d\overline w^i}\overline{\Gamma_-(w;f)}$ is a bounded linear functional for all $i=0,1,2,\dots$.
\end{proof}

\begin{proposition}
Let $r\in \Rat(D)$. Then 
\begin{align}\label{Kcommutator}
UK_r-K_r U &=r\otimes \phi,\\
K^*_r(\phi) &=0.\label{Kadjoint}
\end{align}
Furthermore, these two equations determine $K_r$
uniquely for given $r\in \Rat(D)$ and $\phi\in H^2$.
\end{proposition}
\begin{proof}

The verification  of \eqref{Kcommutator} and \eqref{Kadjoint} is straightforward, although somewhat cumbersome. We will sketch the computations here and leave the details to the reader:
The following formula is well known and easily established: $UT_l-T_lU=(Ul)\otimes 1$ for all $l\in H^2$. 
Thus, 
\[
UT_{zr}T_{\overline\phi}-T_{zr}T_{\overline\phi}U=(UT_{zr}-T_{zr}U)T_{\overline\phi}=r\otimes \phi.
\]
It can be shown by  a similar computation   that the operator $T_{zT_{\overline \phi}r}-T_{ zr\overline{\phi}}$
commutes with $U$. Since $K_r=T_{zr}T_{\overline\phi}+(T_{zT_{\overline \phi}r}-T_{zr\overline{\phi}})$,
we get \eqref{Kcommutator}.


In order to prove  \eqref{Kadjoint}, we first compute that
$K_r^*f=\phi P_+(\overline{zr}f)-P_+(\overline{zr}\phi)f$, for all $f\in H^2$.
Then  $K_r\phi=\phi P_+(\overline{zr}\phi) -P_+(\overline{zr}\phi) \phi=0$.

Finally, let us show that \eqref{Kcommutator} and \eqref{Kadjoint}  determine $K_r$ uniquely. Suppose that
$K'$ is a bounded operator that satisfies these equations. Then $C:=K'-K_r$ commutes with $U$ and satisfies $C^*\phi=0$. Since  $C$ commutes with $U$, we have $C=T_{\overline l}$, with $l\in H^\infty$. So $C^*=T_l$ is multiplication by $l$. But then we cannot have $C^*\phi=0$
unless $l=0$ (since $\phi\neq 0$). We conclude that $C=0$, i.e., $K'=K_r$.
\end{proof}

\begin{proposition}\label{prop:timesalgebra}
$\Rat(D)$ is a commutative algebra under the multiplication $\times$ and standard addition and scalar multiplication. The map  $r\mapsto K_r$ is a representation of this algebra
by operators acting on $H^2$. 
\end{proposition}
\begin{proof}
It is clear that $\times$ is bilinear and commutative. Let us show that it is associative.
It is easily verified, using \eqref{Kcommutator} and \eqref{Kadjoint}, that the operators 
 $K_r K_s$  and $K_{r\times s}$ have the same commutator with $U$ (equal to $(r\times s)\otimes \phi$) and that their adjoints vanish at $\phi$. We conclude by the previous proposition
 that $K_r K_s=K_{r\times s}$. This means that  $r\times(s\times f)=(r\times s)\times f$ for all $f\in H^2$. In particular,
$\times$ is associative. Thus, $(\mathcal R(D),\times)$ is a commutative algebra over $\C$.
Since $K_r$ depends linearly on $r$,  $r\mapsto K_r$ is an algebra
homomorphism. 
\end{proof}

We will use the notation  $\Rat^\times(D)$ to refer to $\Rat(D)$ regarded as an algebra under $\times$. 
Observe that for $\phi=1$ we get $r\times s=zrs$, and so $r\mapsto zr$ is an isomorphism
from $\Rat^\times(D)$ to $z\Rat(D)$ (where the latter is endowed with the standard multiplication).
In the next section we will elucidate the structure of $\Rat^\times(D)$ for an arbitrary $\phi$.

Consider on $\Rat^\times(D)$ the binary operation
\[
r\circ s=r+s-r\times s.
\]
\begin{proposition}
We have
\begin{align*}
(I-K_r)(U+s\otimes \phi)=(U+(r\circ s)\otimes \phi)(I-K_r).
\end{align*}
\end{proposition}
\begin{proof} This follows at once from \eqref{Kcommutator} and \eqref{Kadjoint}. 
\end{proof}

The preceding proposition  implies that if $I-K_r$ is invertible then $U+r\otimes \phi$
and $U+(r\circ s)\otimes \phi$ are similar. We have $(I-K_r)(I-K_s)=I-K_{r\circ s}$
and $I-K_0=I$. So, if $r$ is  an invertible element  of $\Rat^\times(D)$ with respect to  $\circ$ (where 0
is the neutral element)--i.e.,  $r\circ s=0$ for some $s\in \Rat(D)$--then  $(I-K_r)(I-K_s)=I$, and so
$I-K_r$ is invertible.
In general, given a ring $R$ the operation $a\circ b\mapsto a+b-ab$ is called
circle composition and the collection of invertible elements with respect to this
operation is called the circle group of the ring.  We arrive to the following corollary:

\begin{corollary}\label{cor:circlesimilar}
Let $r,s\in \Rat(D)$.
If there exists a circle invertible element $t\in \Rat^\times(D)$ such that $r\circ t=s$ then
$U+r\otimes \phi$ and $U+s\otimes \phi$ are similar.
\end{corollary}

\section{The algebra $\Rat^\times(D)$}\label{timesalgebra}
In this section we elucidate the structure of the algebra $\Rat^\times (D)$. We then show that 
$r\circ t=s$, for some circle invertible  $t$,  if and only if the conditions of Theorem \ref{main} (ii) hold.
Together with Corollary \ref{cor:circlesimilar}, this proves Theorem \ref{main} (ii)$\Rightarrow$(i).

\begin{lemma}
The map $\gamma_+\colon \Rat^\times(D)\to \Rat(D)$ defined by $\gamma_+(r)=zT_{\overline\phi}(r)$ is an algebra homomorphism (where $\Rat(D)$ is endowed with the standard multiplication).
\end{lemma}
\begin{proof}
The map $\gamma_+$ may be viewed as the composition $r\mapsto K_r\mapsto zT_{\overline\phi}(r)$.
As shown in Proposition \ref{prop:timesalgebra}, $r\mapsto K_r$ is an algebra  homomorphism.
On the other hand, $K_r$ is an operator in the Toeplitz algebra and has symbol $zT_{\overline\phi}(r)$
(by Proposition \ref{prop:Kr}). Thus,
$K_r\mapsto zT_{\overline\phi}(r)$ is simply the symbol map. It follows that $\gamma_+$ is an algebra homomorphism.
\end{proof}
Observe that $\Gamma_+(w;r)$--as defined in the introduction--is simply the analytic extension
of $\gamma_+(r)$--as defined in the proposition above--to the interior of the unit disc. 

The range of the homomorphism $\gamma_+$ is  $z\Rat(D)$ (because $T_{\overline\phi}$ maps $\Rat(D)$ onto itself, which in turn can be deduced from \eqref{evalTphi}). So we get
a short exact sequence
\begin{align}\label{shortexact}
0\longrightarrow \ker\gamma_+\longrightarrow\Rat^\times(D)\stackrel{\gamma_+}{\longrightarrow} z\Rat(D)
\longrightarrow 0.
\end{align}
We will show below that this short exact sequence splits, and so $\Rat^\times(D)\cong \ker\gamma_+\oplus z\Rat(D)$. But first, let us investigate the ideal $\ker\gamma_+$ further.

We have $\ker \gamma_+=(\ker T_{\overline\phi})\cap \Rat(D)$.
From \eqref{evalTphi}, we see that $k_a^{(n)}\in \ker T_{\overline\phi}$ if and only if $a$ is a zero of $\phi$ of order larger than $n$. For each $|a|<1$ and $N\in \N$, let $\mathcal S_a^{N}$ denote the linear span of $k_a^{(j)}$, with $j=0,\dots,N-1$.
Then 
\begin{align}\label{kergamma}
\ker\gamma_+=\ker T_{\overline\phi}\cap \Rat(D)=\bigoplus_{\{|a|<1\mid\phi(a)=0\}} S_a^{N_a},
\end{align}
where the direct sum is taken over the zeros of $\phi$ and $N_a$ denotes the order of the zero.  
By Lemma \ref{lemma:deltaformula}, $k_a^{(m-1)}\times r\in \mathcal S_a^m$ if $a$ is a zero of $\phi$
and $m\leq N_a$. It follows that, in this case, 
$\mathcal S_a^{m}$ is an ideal of $\Rat^\times(D)$. Thus, the direct sum in \eqref{kergamma}
holds  in the ring theoretic sense, i.e., the different summands are orthogonal to each other with respect to $\times$ (indeed, $S_a^{N_a}\times S_b^{N_b}\subseteq S_a^{N_a}\cap S_b^{N_b}=\{0\}$ if $a\neq b$).

Let $a$ be a zero of $\phi$.
Let $u_a=(\frac{z-a}{1-\overline{a}z})^{N_a}$. Let $\psi\in H^2$ denote the function such that $\phi=u_a\psi$. Observe that $u_a$ and $\psi$ are relatively prime, since $a$ is a zero of order $N_a$
of $\phi$. Thus, there exists $\alpha,\beta\in H^2$ such that $\alpha\psi-u_a\beta=1$.
In fact, we can choose $\alpha$ a rational function in $(u_aH^2)^\perp$.
Define $e_a\in \mathcal S_a^{N_a}$ by $e_a=P_+(\overline{z\alpha}u_a)$.

\begin{lemma}
Let $a$ be a zero of $\phi$.
\begin{enumerate}[(i)]
\item
If $N_a\geq 2$ the map $k_a^{(N_a-2)}\mapsto x$ extends to an algebra  isomorphism  from  $\mathcal S_a^{N_a-1}$ to $\C[x]/(x^{N_a})$.
\item
The map $e_a\mapsto 1$, $k_a^{(N_a-2)}\mapsto x$ extends to an algebra isomorphism from $\mathcal S_a^{N_a}$ to the unitization of $\C[x]/(x^{N_a})$.
\end{enumerate}
\end{lemma}

\begin{proof}
(i) It suffices to show that the elements $\displaystyle{\times_{i=1}^m k_a^{(N_a-2)}}$, with $m=1,2,\dots,N_a-1$, span $\mathcal S_a^{N_a-1}$ and that $\displaystyle{\times_{i=1}^{N_a} k_a^{(N_a-2)}}=0$. These assertions, in turn, will follow once we have shown that 
\begin{enumerate}
\item
$k_a^{(N_a-2)}\times k_a^{(m)}\in \mathcal \mathcal S^{m-1}_a\backslash \mathcal S^{m-2}_a$ for all $m=1,2,\dots,N_a-1$, and 
\item
$k_a^{(N_a-2)}\times k_a=0$.
\end{enumerate}
From Lemma \ref{lemma:deltaformula}, we have
\[
k_a^{(N_a-2)}\times k^{(m)}_a=\overline{\Gamma_-( a;k_a^{(N_a-2)})}k^{(m)}_a+\frac{d}{d\overline{w}}\overline {\Gamma_-(w;k_a^{(N_a-2)})}|_{w=a}\cdot k^{(m-1)}_a+\dots.
\]
Both (1) and (2) above  follow from $\Gamma_-(\overline a;k_a^{(N_a-2)})=0$ and 
$\frac{d}{dw}\Gamma_-( w;k_a^{(N_a-2)})|_{w=a}\neq 0$.

(ii) Since $\mathcal S_a^{N_a}/\mathcal S_a^{N_a-1}$ is one dimensional, it suffices to show that
$e_a$ is a unit of $\mathcal S_a^{N_a}$ (and use (i)). 

We must show that $e_a\times k_a^{(m)}=k_a^{(m)}$ for all $m\leq N_a$. By Lemma \ref{lemma:deltaformula}, this is equivalent to $\Gamma_-(a;e_a)=1$ and $\frac{d^m}{d w^m}\Gamma_-(w;e_a)|_{w=a}=0$ for all $m=1,2,\dots,N_a$. Thus, we must show that 
$a$ is a zero of $\Gamma_-(w;e_a)-1$ of order at least  $N_a$. Let us compute $\Gamma_-(w;e_a)$:
\begin{align*}
\Gamma_-(w;e_a) &=\langle  \frac{\phi}{w-z},e_a\rangle
=\langle \frac{\phi}{w-z},\overline{z\alpha}u_a\rangle=\langle\frac{\psi}{w-z},\overline {z\alpha}\rangle =
\langle\alpha\psi,\frac{1}{1-z\overline w}\rangle\\
&=\langle1+u_a\beta,\frac{1}{1-z\overline w}\rangle=1+ u_a(w)\beta(w)
=1+\Big(\frac{w-a}{1-\overline{a}w}\Big)^{N_a}\beta(w).\qedhere
\end{align*}
\end{proof}
%

 Recall the exact sequence \eqref{shortexact}. We can now conclude that this sequence splits, since
the map
 $\gamma_-\colon \Rat^\times(D)\to \ker \gamma_+$ defined by
\[
\gamma_-(r)=\sum_{\{|a|<1\mid\phi(a)=0\}} r\times e_a\] 
 is a right inverse of the inclusion $\ker \gamma_+\hookrightarrow \Rat^\times (D)$.
Thus, $\Rat^\times(D)\cong \ker\gamma_+\oplus z\Rat(D)$.
We summarize our findings in the following theorem:

\begin{theorem}\label{th:timesalgebra}
The following propositions are true.
\begin{enumerate}[(i)]
\item
$\Rat^\times(D)\stackrel{(\gamma_-,\gamma_+)}\longrightarrow \ker\gamma_+\oplus z\Rat(D)$
is an isomorphism.
\item
$\ker\gamma_+=\bigoplus_{\{|a|<1\mid \phi(a)=0\}} S_a^{N_a}$, where the projection onto the $a$-th summand is given by $r\mapsto r\times e_a$.
\item
For each $|a|<1$, zero of $\phi$, there is an isomorphism from $S_a^{N_a}$ to $(\C[x]/(x^{N_a}))^\sim$ such that  $e_a\mapsto 1$ and $k_a^{(N_a-2)}\mapsto x$. Here $(\C[x]/(x^{N_a}))^\sim$ denotes the unitization of
$\C[x]/(x^{N_a})$.
\end{enumerate}
\end{theorem}

We are now ready to describe when two elements of $\Rat^\times(D)$ are in the same orbit of the action of the group of circle invertible elements.

\begin{corollary}\label{cor:circlegroup}
\begin{enumerate}[(i)]
\item
The element $t\in \Rat^\times (D)$ is circle invertible if and only if  $1-\gamma_+(t)$ is invertible in $\Rat(D)$
and $\Gamma_-(a;t)\neq 1$ for any $|a|<1$, zero of $\phi$.
\item
Let $r,s\in \Rat^\times(D)$. There exists a circle invertible element $t\in \Rat^\times(D)$
such that $r\circ t=s$  if and only if $r$ and $s$ satisfy conditions (a) and (b) of Theorem \ref{main}. 
\end{enumerate}
\end{corollary}
\begin{proof}
(i) By the previous theorem, $t\in \Rat^\times(D)$ is circle invertible if $\gamma_-(t)\in \ker \gamma_+$ and $\gamma_+(t)\in z\Rat(D)$ are circle invertible. The latter condition is equivalent  to $1-\gamma_+(t)$ being invertible in $\Rat(D)$, while
the former is equivalent to $e_a-e_a\times t\in \mathcal S_a^{N_a}$ being invertible for all $|a|<1$, zero of $\phi$. An element
of $(\C[x]/(x^{N_a}))^\sim$ is invertible if and only if it is not in the nil ideal $\C[x]/(x^{N_a})$. Applied to $\mathcal S_a^{N_a}$, this is equivalent
to the coefficient of $k_a^{(N_a-1)}$ in $k_a^{N_a-1}\times t$ being different from  1. By Lemma \ref{lemma:deltaformula}, this is the same as $\Gamma_-(a;t)\neq 1$.

(ii) In order for $r$ and $s$ to be related by circle invertible elements, we must have that
\begin{enumerate}
\item
$\gamma_+(r)$ and $\gamma_+(s)$ are related by a circle invertible element of $z\Rat(D)$,
\item
for each $|a|<1$, zero if $\phi$,  $e_a\times r$ and $e_a\times t$ are related by a circle invertible element of $S_a^{N_a}$.
\end{enumerate} 
The first condition is equivalent to $1-\gamma_+(r)$ and $1-\gamma_+(s)$ differing by an invertible factor
of $\Rat(D)$. This is equivalent to condition (a) of Theorem \ref{main} (ii). 
The second condition is equivalent to $e_a-e_a\times r$ and $e_a-e_a\times s$ being both invertible or having the same order of nilpotency for each $a$ zero of $\phi$. (This criterion is easily verified in $(\C[x]/(x^{N_a}))^\sim$.) In turn, this is equivalent to 
\[
k_a^{(N_a-1)}-k_a^{(N_a-1)}\times r\in \mathcal S_a^m\Leftrightarrow k_a^{(N_a-1)}-k_a^{(N_a-1)}\times s\in \mathcal S_a^m,
\]
for all $m=1,2,\dots,N_a$. By Lemma \ref{lemma:deltaformula}, this is equivalent to condition (b) of Theorem \ref{main} (ii). 
%
\end{proof}

\begin{proof}[Proof of Theorem \ref{main} (ii)$\Rightarrow$(i).]This follows at once from Corollary \ref{cor:circlesimilar} (ii) and Corollary \ref{cor:circlegroup}.
\end{proof}

\section{Proof of (i) implies (ii)}\label{pointspectrum}

In this section we prove the implication (i)$\Rightarrow$(i) of Theorem \ref{main}. We start with a lemma.

\begin{lemma}
Let $A$ be a bounded  operator on a Hilbert space and let $B$ be a left inverse for $A$, i.e., $BA=I$. 
Let $\tilde A=A+f\otimes g$.
\begin{enumerate}[(i)]
\item
We have $\ker \tilde A\neq 0$ if and only if 
$ABf=f$ and 
$1+\langle Bf,g\rangle=0$. In this case  $\ker \tilde A=\mathrm{span}\{Bf\}$.
\item
Assume that $\ker \tilde A\neq 0$.
For $k>1$ we have that $\ker \tilde A^k\neq \ker \tilde A^{k-1}$ if and only if
$AB^if=B^{i-1}f$ for $1<i\leq k$ and 
$\langle B^if,g\rangle=0$ for $1<i\leq k$.
In this case $\ker \tilde A^k=\mathrm{span}\{B^if\mid 1\leq i\leq k\}$.
\end{enumerate}
\end{lemma}

\begin{proof}
(i) This is a straightforward computation (left to the reader).

(ii) Since $\ker \tilde A$ has dimension 1 (by (i)), the dimension of $\ker \tilde A^k$, for $k=1,2,\dots$, grows by 1 and then becomes stationary.
So $\dim \ker \tilde A^k\leq k$. If $\langle B^if,g\rangle=0$ for $1<i\leq k$ (and -1 for $i=1$) and $AB^if=B^{i-1}f$
then we easily verify that $\mathrm{span}\{B^if\mid 1\leq i\leq k\}\subset \ker \tilde A^k$. Also, the vectors on the left side are linearly independent (they form a Jordan chain). So we must have equality.
This also shows that $\ker \tilde A^{k-1}\neq \ker \tilde A^k$.

We will prove the other implication by induction on $k$. Assume it is true for $k$. 
Suppose that $\ker \tilde A^{k+1}\neq \ker \tilde A^k$. Since $\tilde A$ maps $\ker \tilde A^{k+1}$
surjectively onto $\ker \tilde A^k$, there exists $x$ such that $\tilde Ax=B^kf$. That is,
$Ax+f\langle x,g\rangle=B^kf$. Multiplying by $B$ we get $x+Bf\langle x,g\rangle=B^{k+1}f$.
It follows that $\tilde AB^{k+1}f=B^kf$. This in turn implies that $AB^{k+1}f+f\langle B^{k+1}f,g\rangle=B^kf$.
Multiplying by $B$ and using that $Bf\neq 0$ we get $\langle B^{k+1}f,g\rangle=0$. Then $AB^{k+1}f=B^kf$.     
\end{proof}

\begin{proposition}
Let $r\in \Rat(D)$ and $\phi\in H^2$. Set $U+r\otimes \phi=U_r$.
\begin{enumerate}[(i)]
\item
Let $|w|\leq 1$. Then 
$\dim \ker (1-wU_r)^k=\min (k,\mathrm{ord}_w(1-\Gamma_+(w;r))$.

\item
Let $|w|<1$. Then $
\dim \ker (U_r^*-w)^k=\min (k,\mathrm{ord}_w(\phi),\mathrm{ord}_w(1-\Gamma_-(w;r)))$.
\end{enumerate}
\end{proposition}

\begin{proof}
(i) We have that $1-wU_r=(1-wU)-r\otimes \overline{w}\phi$. Assume first that $|w|<1$. We can apply the previous lemma with $A=1-wU$ and
 $B=(1-wU)^{-1}$. We get that $\dim \ker (1-cU^{\rho})^k=k$
if and only if  $\langle (1-wU)^{-1}r,\overline{w}\phi\rangle=1$ and $\langle (1-wU)^{-i}r,\overline{w}\phi\rangle=0$
for $1<i\leq k$. 
This leads to  $\mathrm{ord}_w(1-\Gamma_+(w;r))\geq k$. 

The case $|w|=1$ can be handled similarly. In this case
we set $A=1-wU$ and $B=T_{\frac{1}{1-w\overline z}}$. Observe that, although $B$ is not bounded, it maps $\Rat(D)$
surjectively onto itself. Also, $A$ maps $\Rat(D)$ into itself,  and $BAf=f$ for all $f\Rat(D)$. This makes the computations of the previous lemma still applicable, since it is easy to check that in this case $\ker \tilde A^k\subseteq \Rat(D)$, for all $k$. So we may restrict our computations to $\Rat(D)$ from the outset.

(ii) We have that $U_r^*-wI=(U^*-wI)+\phi\otimes r$. Thus, we can apply the previous lemma with $A=U^*-wI$ and 
$B=T_{\frac{1}{z-w}}$. We get that $\dim \ker (U_r^*-w)^k=k$ if and only if 
\begin{enumerate}
\item
$(U^*-wI)(T_{\frac{1}{z-w}})^i\phi=(T_{\frac{1}{z-w}})^{i-1}\phi$ for $1\leq i\leq k$,
\item 
$\langle T_{\frac{1}{z-w}}\phi,r\rangle=1$, and  $\langle (T_{\frac{1}{z-w}})^i\phi,r\rangle=0$ for $1<i\leq k$. 
\end{enumerate}
The first condition is satisfied if and only if $\mathrm{ord}_w(\phi)\geq k$, and the second if and only if  $\mathrm{ord}_w(1-\Gamma_-(w;r))\geq k$. This proves (ii).
\end{proof}

\begin{proof}[Proof of Theorem \ref{main}(i) $\Rightarrow$ (ii)]
For each $k=1,2,\dots$, the quantities $\dim \ker (1-wU_r)^k$ and $\dim \ker (U_r^*-w)^k$  are similarity invariants.
This, combined with the previous proposition, proves the implication (i) $\Rightarrow$ (ii) in Theorem \ref{main}.
\end{proof}

\end{document}